\documentclass[a4,11pt]{amsart}
\usepackage{amssymb,amsmath,amsthm,txfonts}
\usepackage{cite}

\usepackage{color}

\newtheorem{thm}{Theorem}[section]
\newtheorem{lem}[thm]{Lemma}

\newtheorem{prop}[thm]{Proposition}
\theoremstyle{definition}

\theoremstyle{remark}

\newtheorem{rems}[thm]{\textbf{Remarks}}

 \makeatletter
    
    \@addtoreset{equation}{section}
  \makeatother

\makeatletter
      \def\@makefnmark{%
         \leavevmode
            \raise.9ex\hbox{\check@mathfonts
                \fontsize\sf@size\z@\normalfont%
                            \@thefnmark}%
       }
      \makeatother

\begin{document}

\title[]{Rigidity of Beltrami fields with a non-constant proportionality factor}
\author[]{Ken Abe}
\date{}
\address[K. ABE]{Department of Mathematics, Graduate School of Science, Osaka City University, 3-3-138 Sugimoto, Sumiyoshi-ku Osaka, 558-8585, Japan}
\email{kabe@osaka-cu.ac.jp}

\date{\today}

\maketitle

\begin{abstract}
We prove that bounded Beltrami fields must be symmetric if a proportionality factor depends on 2 variables in the cylindrical coordinate and admits a regular level set diffeomorphic to a cylinder or a torus. 
\end{abstract}

\section{Introduction}
We consider 3d steady states of ideal incompressible flows

\begin{align*}
(\nabla \times u)\times u+\nabla \pi=0,\quad \nabla \cdot u=0\quad \textrm{in}\ \mathbb{R}^{3},\tag{1.1}
\end{align*}\\
where $u$ is the velocity of fluid and $\pi$ is the Bernoulli pressure. Integral curves of the velocity and the vorticity $\nabla \times u$ are called stream lines and vortex lines, respectively. If the Bernoulli pressure $\pi\nequiv \textrm{const.}$ is regular, they lie on level sets of $\pi$, called Bernoulli surfaces. It is known \cite{AK98} that Bernoulli surfaces are diffeomorphic to nested cylinders or tori. The system (1.1) can be written as an elliptic system with constraints, e.g. \cite[p.34]{Grad85}. Indeed, by introducing the current potential $\eta$ such that $\nabla \times u=\nabla \pi\times \nabla \eta$, (1.1) is formally written in an equivalent form

\begin{equation*}
\begin{aligned}
&\nabla \times u=\nabla \pi\times \nabla \eta,\quad \nabla \cdot u=0,\\
&u\cdot \nabla \pi=0,\quad u\cdot \nabla \eta=1.
\end{aligned}
\end{equation*}\\
The first line is an elliptic system for given $\pi$ and $\eta$. The second line is constraints to them, called a degenerate hyperbolic system. 

The constraints are removed by symmetry, e.g. translation or rotation. In the axisymmetric setting, (1.1) is reduced to the Grad-Shafranov equation \cite{Grad}, \cite{Shafranov}. Existence of solutions with compactly supported vorticity is well known in the study of vortex rings, e.g. \cite{Van13}. Moreover, compactly supported solutions are constructed in \cite{Gav}, \cite{CLV}, \cite{DEPS21}. Existence of smooth non-symmetric solutions to (1.1) with $\pi\nequiv \textrm{const.}$ is unknown.

The non-existence of such non-symmetric solutions is a conjecture of Grad \cite[p.144]{Grad67}, see Constantin et al. \cite[p.529]{CDG21b}. More precisely, symmetries in this conjecture are 4 types: translation, rotation, helix, reflection in a plane. This problem is considered as \textit{rigidity} to (1.1). For the 2d flows, a rigidity result that bounded solutions with no stagnation points must be shear flows is proved by Hamel and Nadirashvili \cite{HN19}. See also rigidity in a strip \cite{HN17} and in a pipe for axisymmetric flows \cite{CDG21b}. The full 3d rigidity to (1.1) with $\pi\nequiv \textrm{const.}$ is unknown, cf. \cite{Shv}. Grad's conjecture is also studied from existence of non-symmetric solutions with piecewise constant pressure \cite{BL96}, \cite{ELP21} and of smooth non-axisymmetric solutions \cite{BKM20}, \cite{BKM20b}, \cite{CDG21}.\\ 

In this paper, we study rigidity of (1.1) with constant pressure $\pi\equiv \textrm{const.}$ Velocity and vorticity with such the pressure are collinear and (1.1) is reduced to the Beltrami equations

\begin{align*}
\nabla \times u=f u,\quad \nabla \cdot u=0\quad \textrm{in}\ \mathbb{R}^{3}.\tag{1.2}
\end{align*}\\
The function $f$ is called a proportionality factor. If $f\equiv \textrm{const.}$, $u$ is called a strong Beltrami field. Vortex lines of strong Beltrami fields can be chaotic  and non-symmetric, e.g. ABC flows \cite{AK98}. Hence (1.2) with $f\equiv \textrm{const}$. is \textit{not} rigid. It is known \cite{EP12}, \cite{EP15} that strong Beltrami fields describe knots and links of vortex lines and vortex tubes. 

If $f\nequiv \textrm{const.}$, vortex lines are confined to a level set $f^{-1}(c)=\{x\in \mathbb{R}^{3}\ |\ f(x)=c  \}$ for $c\in \mathbb{R}$ since $f$ is a first integral, i.e. 

\begin{align*}
u\cdot \nabla f=0.
\end{align*}\\
It is known \cite{AK98} that a closed surface $f^{-1}(c)$ with no singular points $\{u=0\}$ is diffeomorphic to a torus. Existence of solutions to (1.2) is unknown unless $f\equiv \textrm{const}$ or under symmetry, cf. \cite{CK20}. Axisymmetric solutions with compactly supported vorticity exist \cite{Chandra}, \cite{Tu89}, \cite{A8}.

In contrast to (1.1) with $\pi\nequiv \textrm{const.}$, rigidity results are known to (1.2). The first rigidity results to (1.2) are Liouville theorems on decay conditions at space infinity \cite{Na14}, \cite{CC15}, e.g. $u=o(|x|^{-1})$ as $|x|\to\infty$. This decay rate is sharp, cf. \cite{EP12}, \cite{EP15}. The another type Liouville theorem is based on a level set condition for $f \nequiv \textrm{const}$. 

\vspace{5pt}

\begin{thm}[\cite{EP16}]
Suppose that $f\in C^{2+\mu}(\mathbb{R}^{3})$ for some $0<\mu<1$. If $f$ admits a regular level set diffeomorphic to a sphere, then any solutions to (1.2) is identically zero. 
\end{thm}

\vspace{5pt}

This Liouville theorem implies non-existence to (1.2) for a broad class of $f$, e.g. radial or having extrema. On the other hand, it implies a certain relation between existence and symmetry of $f$ since symmetric $f$ does not take extrema. The relation between existence and symmetry is indeed Grad's conjecture to (1.1) with $\pi\nequiv\textrm{const.}$

For symmetric $f$ depending on 1 variable in the cylindrical coordinate $(r,\theta,z)$, i.e. $x_1=r\cos\theta$, $x_2=r\sin\theta$, $x_3=z$, any bounded solutions of (1.2) are symmetric, and even for such $f$, solutions may not exist \cite{GPS01}, \cite[Section 5]{CK20}. More precisely, the following 3 cases are known.\\

\noindent
(i) For $f=f(z)$, level sets are planes. Any solutions of (1.2) are harmonic on them and singular points are isolated. In particular, bounded solutions are symmetric, i.e. $u=u(z)$.

\noindent
(ii) For $f=f(r)$, level sets are cylinders. Any solutions of (1.2) are constant on them and axisymmetric. 

\noindent
(iii) For $f=f(\theta)$, level sets are half planes. Any solutions of (1.2) are trivial.\\

This rigidity follows by investigating \textit{compatibility} of a constrained evolution equation equivalent to (1.2), see Remarks 2.1. The constraint in (1.2) is understood as compatibility for the constrained evolution equation \cite{EP16} and studied via an exterior differential system by using Cartan's method \cite{CK20}.

For $f$ depending on 2 variables in $(r,\theta,z)$, a variety of surfaces appear as level sets of $f$ such as cylindrical surfaces for $f=f(r,\theta)$, surfaces of revolution for $f=f(r,z)$ and right conoids for $f=f(\theta,z)$. For $f=f(r,\theta)$ and $f=f(r,z)$ having extrema on $(x_1,x_2)$ or $(r,z)$-planes, their level sets admit a cylinder or a torus, cf. \cite{EP16}.

If only 4 types of symmetry are admitted to (1.2) with $f\nequiv \textrm{const.}$, $f=f(r,\theta)$ and $f=f(r,z)$ are 2 types among them. We establish a rigidity result for these $f$. Since translationally and rotationally symmetric solutions exist, divisions of these 2 cases must appear in rigidity to (1.2).

\vspace{5pt}

\begin{thm}
Suppose that $f\in C^{2+\mu}(\mathbb{R}^{3})$ for some $0<\mu<1$. \\
\noindent
(i) If $f=f(r,\theta)$ admits a regular level set diffeomorphic to a cylinder, then any bounded solutions to (1.2) are translationally symmetric. 

\noindent
(ii) If $f=f(r,z)$ admits a regular level set diffeomorphic to a torus, then any solutions to (1.2) are rotationally symmetric. 
\end{thm}

\begin{rems}
(i) Translationally symmetric solutions of (1.2) form

\begin{align*}
u=\partial_2\Psi e_{1}-\partial_1\Psi e_{2}+u^{3}(\Psi)e_3,\quad f=\dot{u}^{3}(\Psi), 
\end{align*}\\
for some $u^{3}(\cdot)$ and the stream function $\Psi(x_1,x_2)$ satisfying $-\Delta \Psi=\dot{u}^{3}(\Psi){u}^{3}(\Psi)$, where $e_1,e_2,e_3$ is the orthogonal basis in the Cartesian coordinate. 

\noindent
(ii) Rotationally symmetric (axisymmetric) solutions of (1.2) form

\begin{align*}
u=-r^{-1}\partial_z\Psi e_r+r^{-1}\Gamma(\Psi)e_{\theta}+r^{-1}\partial_r\Psi e_z,\quad f=\dot{\Gamma}(\Psi),  
\end{align*}\\
for some $\Gamma(\cdot)$ and $\Psi(r,z)$ satisfying $-(\Delta_{z,r}-r^{-1}\partial_r)\Psi=\dot{\Gamma}(\Psi)\Gamma(\Psi)$, where $e_r={}^{t}(\cos\theta,\sin\theta,0)$, $e_{\theta}={}^{t}(-\sin\theta,\cos\theta,0)$, $e_z={}^{t}(0,0,1)$ are the orthogonal basis in the cylindrical coordinate. These elliptic problems appear as free boundary problems for translating vortex pairs and vortex rings, see Section 4. Under helical symmetry, (1.1) and (1.2) are reduced to the helical Grad-Shafranov equation \cite[p.196]{Freidberg}. 
\end{rems}

\vspace{10pt}
The proof of Theorem 1.2 is based on the facts that (1.2) can be recasted as a constrained evolution equation on a level set of $f$ \cite{EP16}, \cite{CK20} and that Beltrami fields are solutions to the elliptic equation $-\Delta u=\nabla f\times u+f^{2}u$ as explained below. Unfortunately, solutions of (1.1) with $\pi\nequiv \textrm{const.}$ do not possess neither of them and their rigidity is out of reach. Rigidity of (1.1) with $\pi\nequiv \textrm{const.}$ is unknown even for $\pi$ depending on 1 variable in $(r,\theta,z)$. A crucial difference is failure of a unique continuation property to (1.1) with $\pi\nequiv \textrm{const.}$ as compactly supported solutions exist \cite{Gav}, \cite{CLV}, \cite{DEPS21}. \\

We outline the proof of Theorem 1.2. We show that the symmetric-directional component of $u$ ($u^{3}$ or $\Gamma$) is constant on a level set of $f$. This property can be observed from the above form of symmetric solutions. Since $f$, $u^{3}$ and $\Gamma$ are functions of $\Psi$, $u^{3}$ and $\Gamma$ are constants on a level set of $f$. Moreover, other two components of $u$ are independent of the symmetric variable ($z$ or $\theta$).

To prove this, without assuming symmetry of $u$, we use differential forms and rewrite (1.2) as a constrained evolution equation \cite{EP16}, cf. \cite{CK20},

\begin{align*}
\beta_t=-(c+t)\chi*_t\beta,\quad d \beta=0,  \tag{1.3}
\end{align*}\\
for $\chi=|\nabla f|^{-1}$ and a dual 1-form $\beta$ of $u$ on the surface $f^{-1}(c+t)$, where $d$ denotes the exterior derivative and $*_t$ denotes the Hodge star operator on the surface. By parametrizing the surface by $\xi={}^{t}(\xi_1,\xi_2)$ and $t\geq 0$, this 1-form is written as 

\begin{align*}
\beta=\beta_{1}(\xi,t)d\xi_1+\beta_{2}(\xi,t)d\xi_2.
\end{align*}\\
The constrained equation (1.3) is equivalent to (1.2) for $f\nequiv \textrm{const.}$ and implies that the 1-form satisfies the elliptic equation on the surface

\begin{align*}
d(\chi*_{t}\beta)=0,\quad d\beta=0.   \tag{1.4}
\end{align*}\\
If the surface $f^{-1}(c+t)$ is diffeomorphic to a sphere, $\beta$ is an exact form, i.e. $\beta= d\psi$. By the elliptic equation of the divergence form

\begin{align*}
d(\chi*_{t}d \psi)=0,   
\end{align*}\\
and integration by parts, $\psi$ is constant on the surface. Thus $u$ vanishes in a neighborhood of the regular level set and in $\mathbb{R}^{3}$ by unique continuation. 

If the surface is not diffeomorphic to a sphere, the problem (1.4) admits non-trivial solutions and does not imply non-existence to (1.2) for $f\nequiv \textrm{const.}$ A new observation of the present work is that if $f=f(r,\theta)$ or $f=f(r,z)$, each component of $\beta$ can be constant for the symmetric variable. For cylindrical surfaces and surfaces of revolution, we denote the symmetric variable by $\xi_2$, i.e. $\xi_2=z$ or $\xi_2=\theta$. Then the symmetric-directional component $\beta_2$ satisfies 

\begin{align*}
d(\chi*_{t}d \beta_{2})=0.     \tag{1.5}
\end{align*}\\
We show that $\beta_{2}=\beta_{2}(t)$ and $\beta_{1}=\beta_{1}(\xi_1,t)$ if the surface $f^{-1}(c+t)$ is diffeomorphic to a cylinder or a torus. If the surface is diffeomorphic to a torus, this follows from integration by parts. For a cylinder, we apply a Liouville theorem for bounded solutions to the elliptic equation (1.5). The property $\beta_{2}=\beta_{2}(t)$ and $\beta_{1}=\beta_{1}(\xi_1,t)$ implies local symmetry of $u$ and global symmetry follows from unique continuation. 

It is possible to extend this approach for level sets diffeomorpic to other cylindrical surfaces and surfaces of revolution. On the other hand, rigidity of (1.2) for $f=f(\theta,z)$ is unknown unless $f=f(z)$ and $f=f(\theta)$. The equation (1.5) for $f=f(\theta,z)$ is written with the metric tensor of the surface in the same form as those of $f=f(r,\theta)$ and $f=f(r,z)$ though dependence on the parameter $\xi={}^{t}(\xi_1,\xi_2)$ is different. We also address the constrained equation (1.3) for $f=f(\theta,z)$.


\section{The elliptic equation for $\beta_{2}$}


We derive the equation (1.3) by parametrizing the surface $f^{-1}(c+t)$ by $\xi$. The elliptic equation (1.4) has an explicit form with $\xi$ and is written in a simpler form for $f$ depending on 2 variables in $(r,\theta,z)$. We show that $\beta_{2}$ satisfies the equation (1.5) for $f=f(r,\theta)$ and $f=f(r,z)$.


\subsection{The constrained equation}

We assume that a level set $f^{-1}(c)$ for $c\in \mathbb{R}$ is regular in the sense that $f^{-1}(c+t)$ is a smooth surface for $0\leq t\leq t_0$ with some $t_0 >0$ and $\nabla f(x)\neq 0$ for $x\in f^{-1}(c+t)$. We parametrize the surface $f^{-1}(c)$ by $x=\Phi_0(\xi)$ with $\xi={}^{t}(\xi_1,\xi_2)$ and define $\Phi(\xi, t)$ by the flow of $X=\nabla f /|\nabla f|^{2}$, i.e. 

\begin{align*}
&\partial_t \Phi=X(\Phi),\quad t>0, \\
&\Phi(\xi,0)=\Phi_0(\xi).   
\end{align*}\\
The flow $\Phi(\xi,t)$ parametrizes the surface $f^{-1}(c+t)$, i.e. $\Phi(\xi,t)\in f^{-1}(c+t)$. Since $f\in C^{2+\mu}$ for some $0<\mu<1$, $\Phi(\xi,t)$ is $C^{2+\mu}$. We may assume that $\Phi(\cdot,t)$ is defined for $0\leq t\leq t_0$. The equations (1.2) for the dual 1-form $\alpha=\sum_{i=1}^{3}u^{i}dx_i$ of $u=(u^{i})$ are

\begin{align*}
d_{\mathbb{R}^{3}}\alpha=f*_{\mathbb{R}^{3}}\alpha,\quad d_{\mathbb{R}^{3}}*_{\mathbb{R}^{3}}\alpha=0,   \tag{2.1}
\end{align*}\\
where $d_{\mathbb{R}^{3}}$ and $*_{\mathbb{R}^{3}}$ are the exterior derivative and the Hodge star operator in $\mathbb{R}^{3}$, respectively. By the elliptic equation $-\Delta u=\nabla f\times u+f^{2}u$ and $f\in C^{2+\mu}$, $u$ and $\alpha$ are $C^{3+\mu}$. The pullback $\beta=\Phi^{*}\alpha$ by the map $\Phi: (\xi,t)\longmapsto x=\Phi(\xi,t)$ satisfies 

\begin{align*}
d_{\mathbb{R}^{3}}\beta=(c+t)*_{\mathbb{R}^{3}}\beta,\quad d_{\mathbb{R}^{3}}*_{\mathbb{R}^{3}}\beta=0,   \tag{2.2}
\end{align*}\\
With the matrices $F=(\partial_1\Phi,\partial_2\Phi,\partial_t\Phi)$ and $\tilde{F}=|F|F^{-1}$, 

\begin{align*}
\beta&=(u^{1},u^{2},u^{3})F\ {}^{t}(d\xi_1,d\xi_2,dt), \\
*_{\mathbb{R}^{3}} \beta&=(u^{1},u^{2},u^{3}){}^{t}\tilde{F}\ {}^{t}(d\xi_2\wedge dt,dt\wedge d\xi_1,d\xi_1\wedge d\xi_2),
\end{align*}\\
where $|F|$ denotes the determinant of $F$. Since $u\cdot \partial_t\Phi=0$, the pullback $\beta$ is a 1-form on a surface and $C^{1+\mu}$, 

\begin{align*}
\beta
=u(\Phi(\xi,t))\cdot \partial_{1}\Phi d \xi_1+u(\Phi(\xi,t))\cdot \partial_{2}\Phi d \xi_2
=:\beta_1(\xi,t)d\xi_1+\beta_2(\xi,t)d\xi_2.
\end{align*}\\
We write the metric tensor by ${\mathcal{G}}=(\partial_{i}\Phi\cdot \partial_{j}\Phi)_{1\leq i,j\leq 2}$ and $\mathcal{G}^{-1}=(g^{ij})_{1\leq i,j\leq 2}$. Since

\begin{align*}
{}^{t}FF=\left(
\begin{matrix}
{\mathcal{G}} & 0 \\
0& \chi^{2}
\end{matrix}
\right),\quad \chi=|\nabla f|^{-1},
\end{align*}\\
and $(u^{1},u^{2}, u^{3} )=(\beta_1,\beta_2,0)F^{-1}$, the Hodge dual in $\mathbb{R}^{3}$ is 

\begin{align*}
*_{\mathbb{R}^{3}}\ \beta
=\chi |{\mathcal{G}}|^{1/2} ( (\beta_1g^{11}+\beta_2g^{21} )d\xi_2\wedge dt+(\beta_1g^{12}+\beta_2g^{22} ) dt\wedge d\xi_1  ).
\end{align*}\\
Then the equations (2.2) imply

\begin{align*}
&\partial_1\beta_{2}-\partial_2\beta_{1}=0,\\
&\partial_t \beta_{1}=(c+t)\chi|{\mathcal{G}}|^{1/2} (\beta_1g^{12}+\beta_2g^{22}),\\
&\partial_t \beta_{2}=-(c+t)\chi|{\mathcal{G}}|^{1/2} (\beta_1g^{11}+\beta_2g^{21}),\\
&\partial_1(\chi |{\mathcal{G}}|^{1/2} (\beta_1g^{11}+\beta_2g^{21}) )+\partial_2(\chi |{\mathcal{G}}|^{1/2} (\beta_1g^{12}+\beta_2g^{22}) )=0.
\end{align*}\\
The last equation follows from the first 3 equations. They can be written as 

\begin{align*}
v_t=Av,\quad \nabla^{\perp}\cdot v=0,  \tag{2.3}
\end{align*}\\
for $v={}^{t}(v^{1},v^{2})$, $v^{i}=\beta_{i}$ with the matrix 

\begin{align*}
A=(c+t)\chi | {\mathcal{G}} |^{1/2}
\left(
    \begin{array}{cc}
      g^{12} & g^{22} \\
      -g^{11} & -g^{21}  
    \end{array}
  \right),
\end{align*}\\
where $\nabla={}^{t}(\partial_{\xi_1},\partial_{\xi_2})$ and $\nabla^{\perp}={}^{t}(\partial_{\xi_2}, -\partial_{\xi_1})$. Taking the rotation implies that $v$ satisfies the elliptic equation $\nabla^{\perp}\cdot(Av)=0$ and $\nabla^{\perp}\cdot v=0$. With the Hodge star operator on the surface, 

\begin{align*}
*_{t}\ \beta=(\beta_1,\beta_2)|{\mathcal{G}}|^{1/2} {\mathcal{G}}^{-1}\ {}^{t}(d\xi_2,-d\xi_1),
\end{align*}\\
(2.3) is written as (1.3), i.e. 

\begin{align*}
\beta_t=-(c+t)\chi*_{t} \beta,\quad d \beta=0.   
\end{align*}\\
The elliptic equation (1.4) follows by differentiating the first equation by $d$. The system (2.3) is overdetermined in the sense that the irrotational condition is generally \textit{not} compatible with the evolution equation. Regarding (1.2) as a constrained evolution equation originates from \cite{EP12} in which Cauchy-Kowalevski theorem is used to construct strong Beltrami fields for given initial surface and tangential data. 

Clelland and Klotz \cite{CK20} derived a similar evolution equation as (2.3) by using a moving frame and studied (1.2) in terms of an integral manifold to an equivalent exterior differential system by using the Cartan's method. Among other results, they showed that associated integral manifolds are at most 3-dimensional if level sets of $f$ have no umbilic points.


\subsection{Symmetric $f$}

The matrix $A$ has a simpler form for $f$ depending on 2 variables in $(r,\theta,z)$. \\

\noindent
(i) $f=f(r,\theta)$. We parametrize a curve in a plane by $(r(\xi_1,t), \theta(\xi_1,t))$, i.e. $x_1=r\cos\theta$ $x_2=r\sin\theta$. Then, the map

\begin{align*}
\Phi(\xi,t)=re_{r}(\theta)+ze_z    \tag{2.4}
\end{align*}\\
for $(r,\theta,z)=(r(\xi_1,t), \theta(\xi_1,t),\xi_2)$, parametrizes a cylindrical surface. The matrix $A$ forms

\begin{align*}
A=(c+t)\chi \left(
    \begin{array}{cc}
      0 &\nu\\
      -1/\nu & 0  
    \end{array}
  \right)  \tag{2.5}
\end{align*}\\ 
for 

\begin{align*}
\chi=\sqrt{|\partial_tr|^{2}+r^{2}|\partial_t\theta|^{2}},\quad  \nu=\sqrt{|\partial_1r|^{2}+r^{2}|\partial_1\theta|^{2}}.
\end{align*}\\
(ii) $f=f(r,z)$. We parametrize a curve in the $(r,z)$-plane by $(r(\xi_1,t), z(\xi_1,t))$. Then the map (2.4) for $(r,\theta,z)=(r(\xi_1,t),\xi_2,z(\xi_1,t))$, parametrizes a surface of revolution. The matrix $A$ is the same form as (2.5) with different coefficients

\begin{align*}
\chi=\sqrt{|\partial_tr|^{2}+|\partial_tz|^{2}},\quad \nu=\sqrt{{|\partial_1r|^2+|\partial_1z|^2 }}/r.
\end{align*}\\
(iii) $f=f(\theta,z)$. We parametrize a curve in the $(\theta,z)$-plane by $(\theta(\xi_1,t), z(\xi_1,t))$. The map (2.4) for $(r,\theta,z)=(\xi_2,\theta(\xi_1,t), z(\xi_1,t))$ parametrizes a right conoid. The matrix $A$ is the same form as (2.5) with different coefficients

\begin{align*}
\chi=\sqrt{ r^{2}|\partial_t\theta|^{2}+|\partial_tz|^{2}},\quad \nu=\sqrt{r^{2}|\partial_1\theta|^{2}+|\partial_1 z|^{2}}.
\end{align*}\\
In all the cases (i)-(iii), the elliptic problem for $v$ is written as 

\begin{align*}
\nabla\cdot(Bv)=0,\quad \nabla^{\perp}\cdot v=0   \tag{2.6}
\end{align*}\\
with the matrix

\begin{align*}
B=\left(
    \begin{array}{cc}
      p &0\\
      0 & q  
    \end{array}
  \right),\quad p=\chi/\nu,\ q=\chi\nu.
\end{align*}\\
In the cases (i) and (ii), $B$ is independent of $\xi_2$. Hence $\partial_2v=\nabla v^{2}$ and 

\begin{align*}
\nabla \cdot(B\nabla v^{2})=0.
 \tag{2.7}
\end{align*}\\
This is written as $d(\chi*_td \beta_{2})=0$ in terms of the differential form.

If a level set of $f=f(r,\theta)$ is diffeomorphic to a cylinder, the curve $x_1=r(\xi_1,t)\cos\theta(\xi_1,t)$, $x_2=r(\xi_1,t)\sin\theta(\xi_1,t)$ is diffeomorphic to a circle. We suppose $0\leq \xi_1\leq 2\pi$, $\xi_2\in \mathbb{R}$ and regard $v^{2}$ as a periodic solution to (2.7) in $\mathbb{R}^{2}$. Since the level set $f^{-1}(c)$ is regular, $\nabla f\neq 0$ and $\partial_1\Phi\neq 0$ for $0\leq t\leq t_0$ and $\xi_1\in \mathbb{R}$. Thus, $\chi=|\nabla f|^{-1}$ and $\nu=|\partial_1 \Phi|$ are bounded from above and below by positive constants. We take some $\lambda(t)$ and $\Lambda(t)$ such that 

\begin{align*}
0<\lambda(t)\leq p(\xi_1,t),q(\xi_1,t)\leq \Lambda(t),\quad \xi_1\in \mathbb{R},\ 0\leq t\leq t_0.  \tag{2.8}
\end{align*}

\vspace{5pt}

\begin{rems}
\noindent
(i) For $f=f(z)$, (1.2) is written as a constrained evolution equation without using the cylindrical coordinate. Solutions for such $f$ are $u={}^{t}(v,0)$ and $v={}^{t}(v^1,v^2)$ satisfying

\begin{align*}
\partial_3v=-fv^{\perp},\quad \nabla \cdot v=0,\quad \nabla^{\perp}\cdot v=0,
\end{align*}\\
for $v^{\perp}={}^{t}(-v^{2},v^{1})$. This evolution equation is \textit{compatible} with the elliptic constraints. Thus for a given harmonic vector field $v(x_1,x_2,0)$ on $\{x_3=0\}$, one can construct non-symmetric solutions to (1.2) for $f=f(z)$. If $u$ is bounded, $v$ must be constant in $\mathbb{R}^{2}$. Hence any bounded solutions are symmetric, i.e. $u=u^{1}(x_3)e_1+u^{2}(x_3)e_2$.

(ii) For $f=f(r)$, any solutions of (1.2) are axisymmetric due to the compatibility of (2.3). Indeed, we take $r=r(t)$ satisfying $df(r(t))/dt=1$ and $\theta=\xi_1$, $z=\xi_2$. Then, the function $\Phi$ in (2.4) parametrizes a cylinder and $v$ satisfies (2.6) for $B=B(t)$. If $p\equiv 1$, differentiating $\nabla \cdot (Bv)=0$ by $t$ implies 

\begin{align*}
0=\nabla \cdot ( \partial_t B v+B\partial_tv )
=\nabla \cdot ( \partial_t B v )
=\partial_tq \partial_2v^{2}.
\end{align*}\\
By $\nabla \cdot (Bv)=0$, $\partial_i v^{i}=0$ for $i=1,2$. By differentiating each components of $\partial_t v=Av$ by $\xi_1$ and $\xi_2$, $v=v(t)$ follows. If $p\nequiv 1$, applying the same argument to $\nabla\cdot (p^{-1}Bv)=0$ yields $v=v(t)$. Thus, $u=u^{\theta}(r)e_{\theta}+u^{z}(r)e_{z}$. 

\noindent
(iii) For $f=f(\theta)$, no solutions exist to (1.2) due to the \textit{incompatibility} of (2.3). Indeed, we take $\theta=\theta(t)$ satisfying $df(\theta(t))/dt=1$ and $r=\xi_2$, $z=\xi_1$. Then, the function $\Phi$ in (2.4) parametrizes a half plane. The first equation of (2.6) is $\nabla \cdot (\xi_2 v)=0$. By differentiating this by $t$, $\nabla \cdot (\xi_2^{2} v^{\perp})=0$ and $v=0$ follows. Thus, $u\equiv 0$.
\end{rems}

\section{The Liouville theorem}


\subsection{Local symmetry}


If a level set of $f=f(r,z)$ is diffeomorphic to a torus, $d(\chi*_td \beta
_{2})=0$ and integration by parts yield $\beta_{1}=\beta_{1}(\xi_1,t)$ and $\beta_{2}=\beta_{2}(t)$ by the compactness of the level set. We assume the boundedness of $u$ and prove this property when the level set of $f=f(r,\theta)$ is diffeomorphic to a cylinder. 

\vspace{5pt}

\begin{prop}
$\beta_{1}=\beta_{1}(\xi_1,t)$, $\beta_{2}=\beta_{2}(t)$,\ $\xi_1\in \mathbb{R}$,\ $0\leq t\leq t_0$.
\end{prop}


\begin{proof}
Since the diagonal matrix $B$ satisfies the elliptic condition by (2.8) and $v^{2}\in C^{1+\mu}$ is a bounded weak solution to (2.7) for $\xi\in \mathbb{R}^{2}$, applying the Liouville theorem \cite[Corollally 3.12, Theorem 8.20]{GT} implies that $v^{2}$ is constant. By $\nabla^{\perp}\cdot v=0$, $v^{1}$ is independent of $\xi_2$.
\end{proof}


\begin{lem}
The solution $u$ is translationally or rotationally symmetric in some symmetric open set $U\subset \mathbb{R}^{3}$.
\end{lem}


\begin{proof}
In both cases (i) and (ii), $\partial_{1}\Phi$ and $\partial_{2}\Phi$ are orthogonal. Thus $u(\Phi)=u(\Phi(\xi,t))$ satisfies 

\begin{align*}
u(\Phi)
&=(u(\Phi)\cdot \partial_{1}\Phi) \frac{\partial_{1}\Phi}{|\partial_{1}\Phi|^{2}}+(u(\Phi)\cdot \partial_{2}\Phi) \frac{\partial_{2}\Phi}{|\partial_{2}\Phi|^{2}} \\
&=\beta_{1}(\xi_1,t)\frac{\partial_{1}\Phi}{|\partial_{1}\Phi|^{2}}
+\beta_{2}(t)\frac{\partial_{2}\Phi}{|\partial_{2}\Phi|^{2}}.
\end{align*}\\

In the case (i), by differentiating $\Phi(\xi,t)=r(\xi_1,t)e_{r}(\theta(\xi_1,t))+\xi_2e_z$ by $\xi_1$ and $\xi_2$, 

\begin{align*}
u(\Phi)=\frac{\beta_{1}(\xi_1,t)}{|\partial_1r|^{2}+r^{2}|\partial_1\theta|^{2}}(\partial_1r e_r+r\partial_1\theta e_{\theta})+\beta_{2}(t)e_{z}.
\end{align*}\\
The right-hand side is independent of $\xi_2=z$. Thus $u$ is translationally symmetric on the level set $f^{-1}(c+t)$ for $0\leq t\leq t_0$. In particular, $u$ is translationally symmetric in $U=D\times \mathbb{R}$ for some open set $D$ in a plane. 

In the case (ii), by differentiating $\Phi(\xi,t)=r(\xi_1,t)e_{r}(\xi_2)+z(\xi_1,t)e_z$ by $\xi_1$ and $\xi_2$, 

\begin{align*}
u(\Phi)
&=\frac{\beta_{1}(\xi_1,t)}{|\partial_1r|^{2}+|\partial_1z|^{2}}(\partial_1r e_r+\partial_1z e_{z})+\frac{\beta_{2}(t)}{r}e_{\theta}.
\end{align*}\\
Each components in the cylindrical coordinate are independent of $\xi_2=\theta$. Thus $u$ is rotationally symmetric on the level set $f^{-1}(c+t)$ for $0\leq t\leq t_0$. In particular, $u$ is rotationally symmetric in a region $U$ rotated some open set in the $(r,z)$-plane around the $z$-axis.
\end{proof}

\vspace{5pt}

\subsection{Unique continuation}

The local symmetry implies the global symmetry by unique continuation. We use a classical unique continuation result under the boundedness of $|\Delta w|/|w|$, e.g. \cite{Wolff93}.


\begin{prop}
Let $w\in C^{2}(\mathbb{R}^{3})$ satisfy 

\begin{align*}
|\Delta w|\leq C_R|w|\quad \textrm{in}\ \{|x|<R\}, 
\end{align*}\\
for each $R>0$ with some $C_R>0$. Assume that $w$ vanishes in some open set in $\{|x|<R\}$. Then, $w\equiv 0$.
\end{prop}

\begin{proof}[Proof of Theorem 1.2]
For translationally symmetric $u$ in $U$, set $w(x)=u(x)-u(x+\tau e_{z})$ for $\tau\in \mathbb{R}$. Then, $w$ is a Beltrami field with $f$ and vanishes in $U$. Since $-\Delta w=\nabla f\times w+f^{2}w$ in $\mathbb{R}^{3}$, by unique continuation, $w\equiv 0$ in $\mathbb{R}^{3}$. Thus $u$ is translationally symmetric in $\mathbb{R}^{3}$, i.e. $u=u(x_1,x_2)$. 

Similarly, for rotationally symmetric $u$ in $U$, set $w(x)=u(x)-{}^{t}R_{\tau}u(R_{\tau}x)$ with $R_{\tau}=(e_{r}(\tau), e_{\theta}(\tau), e_{z} )$ for $\tau\in [0,2\pi]$. Then, applying unique continuation to $w$ implies that $u$ is rotationally symmetric in $\mathbb{R}^{3}$, i.e. $u=u^{r}(r,z)e_r(\theta)+u^{\theta}(r,z)e_{\theta}(\theta)+u^{z}(r,z)e_z$. This completes the proof.
\end{proof}

\vspace{5pt}

\begin{rems}
(i) Under the translational symmetry, (1.2) is reduced to 

\begin{align*}
\partial_2u^{3}=fu^{1},\quad 
-\partial_1u^{3}=fu^{2},\quad 
\partial_1u^{2}-\partial_2u^{1}=fu^{3},\quad 
\partial_1u^{1}+\partial_2u^{2}=0.
\end{align*}\\
With a stream function $\Psi$, ${}^{t}(u^{1},u^{2})=\nabla^{\perp}\Psi$ and ${}^{t}(u^{1},u^{2})\cdot \nabla u^{3}=0$. Hence $u^{3}=u^{3}(\Psi)$, $f=\dot{u}^{3}(\Psi)$ and $\Psi$ is a solution to $-\Delta \Psi=\dot{u}^{3}(\Psi){u}^{3}(\Psi)$.

\noindent
(ii) Under the rotational symmetry, (1.2) is reduced to 

\begin{align*}
-\partial_zu^{\theta}=fu^{r},\quad \partial_zu^{r}-\partial_ru^{z}=fu^{\theta},\quad \partial_r u^{\theta}+u^{\theta}/r=fu^{z},\quad \partial_ru^{r}+u^{r}/r+\partial_z u^{z}=0.
\end{align*}\\
With a stream function $\Psi$, $ru^{z}=\partial_r\Psi$, $ru^{r}=-\partial_z\Psi$ and ${}^{t}(ru^{z},ru^{r})\cdot \nabla_{z,r}\Gamma=0$ for $\Gamma=ru^{\theta}$. Hence $\Gamma=\Gamma(\Psi)$, $f=\dot{\Gamma}(\Psi)$ and $\Psi$ is a solution to $-(\Delta_{z,r}-r^{-1}\partial_r)\Psi=\dot{\Gamma}(\Psi)\Gamma(\Psi)$. 
\end{rems}

\section{Examples of symmetric solutions}

\vspace{5pt}

We review existence of translationally and rotationally symmetric solutions to (1.2).

\subsection{Vortex pairs}

Translationally symmetric solutions can be constructed by the elliptic problem for given $u^{3}(\cdot)$,

\begin{align*}
-\Delta \Psi=\dot{u}^{3}(\Psi)u^{3}(\Psi)\quad \textrm{in}\ \mathbb{R}^{2}.   
\end{align*}\\
The simplest solutions are rotationally symmetric solutions, i.e. $\Psi=\Psi(r)$. For such $\Psi$, level sets of $f$ are cylinders, i.e. $f=f(r)$. If $\dot{u}^{3}(\Psi)u^{3}(\Psi)$ is compactly supported, the Biot-Savart law implies the decay ${}^{t}(u^{1},u^{2})=O(r^{-1})$ as $r\to\infty$, cf. \cite{Na14}, \cite{CC15}. Besides rotationally symmetric solutions, there exist periodic solutions for $\dot{u}^{3}(t)u^{3}(t)=t$ or $e^{t}$. For such solutions, level sets of $f$ are deformed cylinders in $\mathbb{R}^{3}$ and ${}^{t}(u^{1},u^{2})$ is merely bounded, e.g. \cite[2.2.2]{MaB}.

Variational solutions also exist. A vortex pair is a pair of translating 2 vortices with opposite signs in $\mathbb{R}^{2}$. They are symmetric for the $x_2$-variable and constructed via the half plane problem:

\begin{equation*}
\begin{aligned}
-\Delta \Psi&=\dot{u}^{3}(\Psi){u}^{3}(\Psi)\quad \textrm{in}\ \mathbb{R}^{2}_{+},\\
\Psi&=-\gamma\hspace{49pt} \textrm{on}\ \partial\mathbb{R}^{2}_{+}, \\
\partial_1\Psi\to0,\quad \partial_2\Psi&\to -W\hspace{42pt} \textrm{as}\ x_1^{2}+x_2^{2}\to\infty.
\end{aligned}
\end{equation*}\\
The constant $W>0$ is a speed of a vortex and $\gamma\geq 0$ is a flux measuring a distance from a vortex to the boundary $x_2=0$. A typical choice is $u^{3}(t)=t_{+}^{l}$ for $l>1$ and $t_{+}=\max\{t,0\}$. For such $u^{3}$, variational solutions exist and their vortex is compactly supported in $\mathbb{R}^{2}$ \cite{Yang91}. Level sets of $f$ are two symmetric deformed cylinders in $\mathbb{R}^{3}$ and the decay is ${}^{t}(u^{1},u^{2})=\textrm{const.}+O(r^{-1})$ as $r\to\infty$.

\subsection{Vortex rings}

Rotationally symmetric solutions can be constructed via the elliptic problem for given $\Gamma(\cdot)$: 

\begin{equation*}
\begin{aligned}
-(\Delta_{z,r}-r^{-1}\partial_r)\Psi&=\dot{\Gamma}(\Psi)\Gamma(\Psi)\quad \textrm{in}\ \mathbb{R}^{2}_{+},\\
\Psi&=-\gamma\hspace{41pt} \textrm{on}\ \partial\mathbb{R}^{2}_{+}, \\
r^{-1}\partial_z\Psi\to0,\quad r^{-1}\partial_r\Psi&\to -W\hspace{34pt} \textrm{as}\ z^{2}+r^{2}\to\infty.
\end{aligned}
\tag{4.3}
\end{equation*}\\
For the choice $\Gamma(s)=s_+^{l}$ and $l>1$, variational solutions exist and their vortex is compactly supported in $\mathbb{R}^{3}$ \cite{A8}. Level sets of $f$ are tori in $\mathbb{R}^{3}$ and the decay is $u=\textrm{const.}+O(|x|^{-3})$ as $|x|\to\infty$.


\section*{Acknowledgements}
The author is grateful to Professor Daniel Peralta-Salas for many helpful comments on topology of level sets for Beltrami fields with a non-constant proportionality factor. This work is partially supported by JSPS through the Grant-in-aid for Young Scientist 20K14347, Scientific Research (B) 17H02853 and MEXT Promotion of Distinctive Joint Research Center Program Grant Number JPMXP0619217849.\\

\bibliographystyle{alpha}
\bibliography{ref}

\end{document}